\documentclass[12pt,english,a4paper]{amsart}
\usepackage[latin9]{inputenc}
\usepackage{a4}
\usepackage{amsfonts}
\usepackage{amssymb}
\usepackage{amsmath}
\usepackage{amsthm}
\usepackage{changepage}
\usepackage{nomencl}
\usepackage{geometry}
\usepackage{tikz}
\usepackage{verbatim}
\usetikzlibrary{matrix}
\usepackage[toc,page]{appendix}
\begin{document}
\providecommand{\keywords}[1]{\textbf{\textit{Keywords: }} #1}
\newtheorem{theorem}{Theorem}[section]
\newtheorem{lemma}[theorem]{Lemma}
\newtheorem{prop}[theorem]{Proposition}
\newtheorem{kor}[theorem]{Corollary}
\theoremstyle{definition}
\newtheorem{defi}{Definition}[section]
\theoremstyle{remark}
\newtheorem{remark}{Remark}[section]
\newtheorem{problem}{Problem}[section]
\newtheorem{question}{Question}[section]
\newtheorem{condenum}{Condition}[section]

\newcommand{\cc}{{\mathbb{C}}}   
\newcommand{\ff}{{\mathbb{F}}}  
\newcommand{\nn}{{\mathbb{N}}}   
\newcommand{\qq}{{\mathbb{Q}}}  
\newcommand{\rr}{{\mathbb{R}}}   
\newcommand{\zz}{{\mathbb{Z}}}  
\newcommand{\fp}{{\mathfrak{p}}}

\title[Grunwald problem and specialization]{The Grunwald problem and specialization of families of regular Galois extensions}
\author{Joachim K\"onig}
\email{jkoenig@kaist.ac.kr}
\address{Department of Mathematical Sciences, Kaist, 291 Daehak-ro, Yuseong-gu, Daejeon (South Korea)}
\maketitle
\begin{abstract}
We investigate specializations of infinite families of regular Galois extensions over number fields.
The problem to what extent the local behaviour of specializations of one single regular Galois extension can be prescribed has been investigated by D\`ebes and Ghazi in the unramified case, and by Legrand, Neftin and the author in general.
Here, we generalize these results and give a partial solution to Grunwald problems using Galois extensions arising as specializations of a family of regular Galois extensions.
These are so far the most comprehensive results for groups $G$ over a number field $k$ under the only condition that $G$ occurs regularly as a Galois group over $k$.
\end{abstract}
%
%

\section{Introduction and statement of main results}

{\textit{Grunwald problems}}: Let $G$ be a finite group, $k$ be a number field and $S$ be a finite set of primes of $k$. For each $p\in S$, denote by $k_p$ the completion of $k$ at $p$, and let $F_p|k_p$ 
be a Galois extension with Galois group embedding into $G$. By a Grunwald problem (for $G$ over $k$), 
we mean the following question:
\begin{problem}[Grunwald problem]
\label{prob:grunwald}
Does there exist a Galois extension $F|k$ with group $G$ such that the completion of $F|k$ at a prime of $F$ extending $p$ equals $F_p|k_p$, for each $p\in S$?
\end{problem}

In the case where all $F_p|k_p$ are unramified/ at most tamely ramified, we speak of an unramified/ tamely ramified Grunwald problem. Of course, in the case of existence of a solution field $F|k$, 
the Galois group $D_p:=Gal(F_p|k_p)$ is a subgroup of $G$ (the decomposition group at $p$), and the Galois group $I_p:=Gal(F_p|F_p^{ur})$, with the maximal unramified subextension $F_p^{ur}$ of $F_p|k_p$, 
is a normal subgroup of $D_p$ (the inertia group at $p$). Note that, since the embedding of $F|k$ into the completion $F_p|k_p = F\cdot k_p|k_p$ is only well-defined up to automorphism, the pair $(I_p, D_p)$ 
is well-defined up to conjugation in $G$.

A question related to the above is the following:
\begin{problem}[Grunwald problem, group version]
\label{prob:grunwald_groups}
Given pairs $(I_p,D_p)$ of subgroups of $G$ which occur as inertia and decomposition group of some Galois extension of $k_p$ (for all $p\in S$), does there exist a Galois extension $F|k$ with group $G$ possessing $(I_p,D_p)$ 
as inertia and decomposition group at $p$ (up to conjugation in $G$)?
\end{problem}

We also speak of {\textit{the}} Grunwald problem (for $G$ over $k$), meaning the question of precisely which Grunwald problems (in the sense of Problem \ref{prob:grunwald} or Problem \ref{prob:grunwald_groups}) possess a solution.
Famously, Grunwald problems are known to have a solution for abelian groups, with a possible exception at primes extending the rational prime $2$. This is known as the Grunwald-Wang theorem. Results by Harari (\cite{Har07}) give positive answers for Grunwald problems for groups which are iterated semidirect products of abelian groups (again, upon excluding a finite set of bad primes). Recent work of Harpaz and Wittenberg (\cite{HW18}) gives a positive answer for all supersolvable groups, and thus in particular for all nilpotent groups.
A different direction was exhibited by Saltman, who showed that all Grunwald problems have a positive answer if the group $G$ has a generic Galois extension over $k$ (\cite{S82}).

Finally, we note that solvability of Grunwald problems outside some finite set of primes (depending on $G$) for all finite groups $G$ would follow immediately from a famous conjecture by Colliot-Th\'el\`ene (see Conjecture 3.5.8 in \cite{Serre}).%

{\textit{Main results and structure of the paper}}: A reasonable approach to solve the above problems for large classes of groups is via specialization of $k$-regular Galois extensions with group $G$. This has been used successfully by D\`ebes and Ghazi (\cite{DG}) to solve 
the unramified Grunwald problem for groups occurring as $k$-regular Galois groups, under the condition that the set $S$ of primes is disjoint to some finite set $S_0$ only depending on the given $k$-regular $G$-extension.
On the other hand, it has been shown in \cite{KLN} (Section 6) that the same result cannot be obtained in generality for ramified Grunwald problems.

The main insight of this article is the fact that the obstructions to solving all ramified Grunwald problems for a group $G$ which occur in the case of specializing just one regular extension (or also, finitely many) may vanish
by adding just one more parameter. In particular, we obtain a partial positive answer for ramified Grunwald problems (see Theorem \ref{simultaneous}), via specialization of not just one regular $G$-extension, but 
rather a one-parameter family (with certain extra conditions) of $k$-regular $G$-extensions.

We state here a special case in which the result of Theorem \ref{simultaneous} becomes particularly nice:
\begin{theorem}
\label{weak_version}
Let $E|k(s)(t)$ be a one-parameter family of regular Galois extensions with group $G$, and let $t\mapsto t_i\in k(s)\cup\{\infty\}$ be a $k(s)$-rational branch point of $E|k(s)(t)$.
Let $(I,D)$ be the inertia and decomposition group at $t\mapsto t_i$ in $E|k(s)(t)$, and let $F|k(s)$ be the residue field extension at $t\mapsto t_i$ in $E|k(s)(t)$.\\
Assume that $F|k(s)$ is $k$-regular. Then for every finite set $\mathcal{S}$ of primes of $k$, disjoint from a finite set of ``bad" primes (depending only on $E|k(s)(t)$), and for every tuple $(x_p \in D)_{p \in \mathcal{S}}$,
there exist infinitely many Galois extensions of $k$ with Galois group $G$ whose inertia and decomposition group at $p$ equal $I$ and $\langle I,x_p\rangle$ respectively (for each $p\in \mathcal{S}$).
More precisely, such extensions may be chosen as specializations $E_{s_0,t_0}|k$ of $E|k(s)(t)$.
\end{theorem}
Note that decomposition groups at primes of tame ramification in number field extensions are necessarily metacyclic. Furthermore, from computational evidence,
decomposition groups $D$ in geometric Galois extensions have a tendency to be ``as large as possible", i.e., to contain the full centralizer of $I$ in $G$.
In such a scenario, Theorem \ref{weak_version} then ensures a best possible answer to Problem \ref{prob:grunwald_groups} (outside of a finite
set of bad primes) under the extra assumption that the inertia group remains all of $I$.

Our precise notion of ``family" is given in Def.\ \ref{def:family}; it implies in particular that the members of a family have the 
same inertia canonical invariant (with finitely many exceptions). It is obviously unclear whether, for a given number field $k$, all finite groups possess such families, but there are known theoretical criteria in inverse Galois theory
which yield existence for certain classes of groups.

While it is not surprising that such families should allow for a wider variety of specializations than any single regular extension, this intuition has not previously been quantified (and obviously there are also exceptions to the rule).
Our most general result in this direction is contained in Theorem \ref{dec_gr_families}.

In Sections \ref{ex:psl32} and \ref{ex:elab}, we give sample applications, one for the case of a simple explicitly given family of polynomials, and one more theoretical about elementary-abelian groups.
We use the first example to demonstrate an application of our results to the problem of existence of parametric sets, showing that the existence of a one-parameter family of regular Galois extensions with group
$G$ and with some mild technical assumptions already prevents the existence of finite parametric sets (see Corollary \ref{psl32_nonparam}).
Recall here that a set $\mathcal{S}$ of $k$-regular Galois extensions of $k(t)$ with group $G$ is called parametric, if every $G$-extension of $k$ arises as a specialization of some element of $\mathcal{S}$. Existence of such sets
was investigated in several previous papers (\cite{Legrand}, \cite{KL}, \cite{KLN}).

We end by stating a strong version of the regular inverse Galois problem which, if true for some group $G$, would imply positive answers to all Grunwald problems for $G$ (via our methods) outside some finite set of primes (depending on $G$), 
over many number fields. See Theorem \ref{hypothetical}.

{\textit{Ideas used in the proofs}}: 
For the proofs of the main results Theorem \ref{dec_gr_families} and Theorem \ref{simultaneous}, we study the behaviour of the residue class field at a branch point in an extension $E/k(s)(t)$ under specialization of the extra parameter $s$,
which more precisely is a reduction of the field of constants. In particular, we use the theorem of Lang-Weil to ensure that, for a given prime $p$ of $k$ (not in some exceptional finite set) and a suitable specializiation $s\mapsto s_0$, 
the Galois groups of the mod-$p$ reduction of these specialized residue class fields can be prescribed to some extent. The reduction to Lang-Weil was also used as a central idea in the solution of unramified Grunwald problems in \cite{DG},
however in the context of mod-$p$ reductions of the underlying $G$-cover itself rather than on the level of residue field extensions.

Then, using the results of \cite{KLN} about the local behaviour of specializations of a {\textit{single}} regular Galois extension, as well as a theorem of Beckmann about ramification in specializations of regular Galois extensions,
we relate the mod-$p$ behaviour of suitable specializations of $E/k(s)(t)$ to the above mod-$p$ reduction of the residue class field at a branch point, as well as to suitable geometric inertia groups.

\section{Prerequisites}
\subsection{Basics about regular Galois extensions}
Let $K$ be a field. A Galois extension $F|K(t)$ is called $K$-regular (in the following simply {\textit {regular}}), if $F\cap \overline{K} = K$. For any $t_0\in K\cup\{\infty\}$ and any place $\mathfrak{p}$ of $F$ extending the $K$-rational place $t\mapsto t_0$, we have a residue field extension $F_{t_0}|K$. This is a Galois extension, not depending on the choice of place $\mathfrak{p}$. We call it the specialization of $F|K(t)$ at $t_0$.

Now let $K$ be of characteristic zero, and let $F|K(t)$ be a $K$-regular Galois extension $F|K(t)$ with group $G$. Such an extension has finitely many branch points $p_1,...,p_r\in \overline{K}\cup\{\infty\}$, and associated to each branch point $p_i$ is a unique conjugacy class $C_i$ of $G$, corresponding to the automorphism $(t-p_i)^{1/e_i}\mapsto \zeta (t-p_i)^{1/e_i}$ of the Laurent series field
$\overline{K}(((t-p_i)^{1/e_i}))$, where $e_i$ is minimal such that $L$ embeds into $\overline{K}(((t-p_i)^{1/e_i}))$, and $\zeta$ is a primitive $e_i$-th root of unity.\footnote{If $p_i=\infty$, one should replace $t-p_i$ by $1/t$.} This $e_i$ is the {\textit {ramification index}} at $p_i$, and equals the order of elements in the class $C_i$. The class tuple $(C_1,...,C_r)$ is called the {\textit {inertia canonical invariant}} of $L|K(t)$, 
and the tuple $((p_1,...,p_r), (C_1,...,C_r))$ the {\textit {ramification structure}}.

Note that $K$-regular Galois extensions $F|K(t)$ with group $G$ correspond one-to-one to $G$-Galois covers $f: X\to\mathbb{P}^1$, defined over $K$, of compact connected Riemann surfaces. In this paper, we will mainly stick with the function field viewpoint (except for a few places with an explicitly geometric flavor, such as Lang-Weil theorem, which make a wording in terms of covers more natural).

\subsection{Ramification and residue fields in specializations}
Let $k$ be a number field, $E|k(t)$ be a regular Galois extension with group $G$, and $t_0\in \mathbb{P}^1(k)$.
We will make extensive use of previous results relating inertia groups, residue fields etc.\ at primes $p$ in the specialized extension $E_{t_0}|k$ to those in the regular extension $E|k(t)$. The case of inertia groups is contained in work of Beckmann (\cite{Beckmann}).

Let $k$ be a number field, $a_0$ be an algebraic integer, $f\in k[X]$ be its minimal polynomial, and $\mathfrak{p}$ be a finite prime of $k$. Define $I_\mathfrak{p}(a,a_0)$ as the multiplicity of $\mathfrak{p}$ in the fractional ideal generated by $f(a)$. Obviously, we have $I_\mathfrak{p}(a,a_0)\ne 0$ only for finitely many prime ideals $\mathfrak{p}$ of $k$. With this notation, we can state an important criterion of Beckmann, relating ramification regular Galois extensions to ramification in specializations.\footnote{Compared with Beckmann's original criterion, we include the assumption that all branch points of a regular extension $E|k(t)$ are algebraic integers. This is always possible via fractional linear transformations in $t$, and eases the notation (in particular because of the definition of intersection multiplicities in Def.\ 1.1 and 4.1 in \cite{Beckmann}, which requires a distinction in cases in general).} See also Theorem I.10.10 in \cite{MM}, which is closer in wording to our version.
\begin{prop}[Beckmann]
\label{beckmann}
Let $k$ be a number field and $E|k(t)$ be a regular Galois extension with Galois group $G$.
Assume that all branch points of $E|k(t)$ are finite and algebraic integers.
Then for all but finitely many primes $\mathfrak{p}$ of $k$ (with the exceptional set depending on $E|k(t)$), the following holds:\\
If $a\in k$ is not a branch point of $E|k(t)$ then the following condition is necessary for $\mathfrak{p}$ to be ramified in the specialization $E_a|k$:
 $$I_{\mathfrak{p}}(a,a_i)>0 \text{ for some (automatically unique) branch point $a_i$.}$$
Furthermore, the inertia group of a prime extending $\mathfrak{p}$ in the specialization $E_a|k$ is then conjugate in $G$ to $\langle\tau^{I_{\mathfrak{p}}(a,a_i)}\rangle$, where $\tau$ is a generator of an inertia subgroup over the branch point $t\mapsto a_i$ of $k(t)$.
\end{prop}

We only note briefly that this theorem remains true for more general situations, and in particular also for the case that $k$ is a function field of characteristic zero (see \cite{Conrad}).

Next, we deal with residue fields and decomposition groups at ramified primes in specializations. For a regular extension $E|k(t)$, a value $t_0\in \mathbb{P}^1(k)$ and a prime $\fp$ of $k$, we use the notation $I_{t_0,\fp}$ and $D_{t_0,\fp}$ for the inertia and decomposition group at (a prime extending) $\fp$ in the residue field extension $E_{t_0}|k$.

The following Proposition contains some of the main results in \cite{KLN} (and occurs there in a somewhat more general setting, namely for more general fields $k$ and more general branch points $t_i$). It relates the residue field, decomposition group etc.\ at ramified primes in specializations to the respective data in the corresponding regular extension. Statements i) to iii) are contained in \cite[Theorem 4.1]{KLN}, whereas statement iv) is in \cite[Theorem 4.4]{KLN}.
\begin{prop}
\label{mainlemma_kln}
Let $k$ be a number field or a function field of characteristic zero. Let $E|k(t)$ be a regular Galois extension with group $G$.

Let 
$t_i\in k\cup\{\infty\}$ be a $k$-rational branch point of $E|k(t)$, 
and let $\fp$ be a prime of $k$, not in some explicit finite set of ``exceptional" primes (depending on $E|k(t)$).

Let $t_0\in k\cup\{\infty\}$ be a non-branch point such that $I_\fp(t_0,t_i)>0$.
Denote by $I$ and $D$ the inertia and decomposition group at (a fixed place extending) $t\mapsto t_i$ in $E|k(t)$.

Then the following hold:

\begin{itemize} 
\item[i)]
 The completion at $\fp$ of $E_{t_i}|k$ is contained in the unramified part of the completion at $\fp$ of $E_{t_0}|k$.\\
In particular, the residue extension at $\fp$ in $E_{t_i}|k$ is contained in the residue extension at $\fp$ in $E_{t_0}|k$. 
\item[ii)] Up to conjugation in $G$, the inertia group $I_{t_0,\fp}$ equals a subgroup of $I$, and the decomposition group $D_{t_0,\fp}$ fulfills $\varphi(D_{t_0,\fp})= D_{t_i,\fp}$, where $\varphi: D\to D/I$ is the canonical epimorphism.  
\item[iii)] In particular, if in addition $I_\fp(t_0,t_i)$ is coprime to $e_i:=|I|$, then $D_{t_0,\fp}$ equals $\varphi^{-1}(D_{t_i,\fp})$. Furthermore, the residue extension (resp.\ completion) at $\fp$ in $E_{t_0}|k$ equals the residue extension (resp.\ completion) at ${\fp}$ in $E_{t_i}|k$.
\item[iv)] Conversely, every Galois extension $F_\fp|k_\fp$ whose Galois group resp.\ inertia group are isomorphic (under the same isomorphism) to $\varphi^{-1}(D_{t_i,\fp})$ resp.\ $I$, occurs as the completion at $\fp$ of $E_{t_0}|k$, for infinitely many $t_0\in k$.
\end{itemize}
\end{prop}
Since the proof of Proposition \ref{mainlemma_kln} is somewhat involved and technical, we refer to the proofs (of the more general version) given in \cite{KLN}.
\section{Families of regular Galois extensions}
\subsection{Definition and first properties}
From now on, always let $k$ denote a number field.
In the following we will treat one-parameter families of regular $G$-Galois extensions of $k(t)$ (for a finite group $G$). We define more precisely what we mean by this.
\begin{defi}
\label{def:family}
Let $s,t$ be independent transcendentals over $k$
Let $E|k(s)(t)$ be a regular Galois extension with group $G$, with branch points $p_1,...,p_r\in \overline{k(s)}\cup \{\infty\}$
%
%
 and inertia canonical invariant $(C_1,...,C_r)$ (where the $C_i$ are non-identity conjugacy classes of $G$). We call $E|k(s)(t)$ a one-parameter family of regular $G$-extensions (with ramification structure $((p_1,...,p_r), (C_1,...,C_r))$).
\end{defi}
Of course it is perfectly reasonable to define $n$-parameter families as well, for any $n\ge 2$, and indeed, the techniques used in the proofs of our main results remain valid for $n\ge 2$ as well. 
The point of our results in Section \ref{sec:hilbertgrunwald} is however that, at least potentially, already consideration of the case $n=1$ may suffice to solve all Grunwald problems for a given group $G$ (away from some fixed finite set of primes). See in particular Section \ref{sec:conj}. 

The following lemma clarifies what happens to the ramification type of families under specialization of the extra parameter $s$. Its assertion also follows easily using Hurwitz spaces and their corresponding universal families of coverings; however we prefer to give a more elementary proof here.
\begin{lemma}
\label{nondeg_spec}
Let $E|k(s)(t)$ be a one-parameter family of regular Galois extensions with group $G$, with ordered branch point set $(p_1,\dots,p_r)$, and let $I_i\le G$ denote the inertia group at the branch point $p_i$, for $=1,\dots, r$.

Then for all but finitely many specializations $s\mapsto s_0\in k$, the extension $E_{s_0}|k(t)$ is regular with Galois group $G$ (up to a canonical isomorphism, independent of $s_0$),
with branch point set $((p_1(s_0),\dots,p_r(s_0))$ and with 
inertia group $I_i$ at the branch point $p_i(s_0)$ 
($i=1,\dots,r$).
 Here the evaluation $p_i(s_0)$ is to be understood in the following way: 
 Set $p_i(s_0):=\infty$ for $p_i=\infty$, and assume from now on that all $p_i$ are $\ne \infty$.
 Set $L:=k(s)(p_1,...,p_r)\subset \overline{k(s)}$. Then $p_i(s_0)$ is the image of $p_i$ under specialization $L_{s_0}|k$. \footnote{Note that this specialization map is unique only up to conjugation in ${\textrm{Gal}}(k)$, and a priori this
 ambiguity exists separately for each $s_0\in k$. See however the proof below on how to render these maps unique using a continuity condition.}
In particular, if $p_i$ is a zero of $f(s,X)\in k(s)[X]$ and $s_0\in k$ is such that $f(s_0,X)$ is defined and separable of maximal degree,
 then $p_i(s_0)$ is a root of $f(s_0,X)$.
\end{lemma}
We call the extension $E_{s_0}|k(t)$ a {\textit {non-degenerate}} member of the family $E|k(s)(t)$, if it fulfills the assertions of Lemma \ref{nondeg_spec}.
\begin{proof}
The place $s\mapsto s_{0}$ of $k(s)$ defines a non-archimedean valuation on $k(s)$. Let $\nu$ be a prolongation to $E\cdot \overline{k(s)}$, such that $\nu(t)=0$. 
Then $\nu$ induces a constant reduction $E_\nu|k(t)$ of $E|k(s)(t)$, mapping a branch point $p_i$ to $p_i(s_0)$ as described in the lemma.

Here, the image $p_i(s_0)$ is independent of the choice of prolongation $\nu$, up to algebraic conjugates over $k$. However, by a special case of Fried's branch cycle lemma, algebraically conjugate branch points
have the same inertia group up to conjugation in $G$. Therefore, the following arguments are independent of the choice of $\nu$.


With the exception of finitely many $s_{0}$, this constant reduction is a good reduction (in the sense of \cite{Deuring}) meaning that the resulting reduced extension is again a regular function field extension, say 
$\widehat{E}|k(t)$, of the same degree as $E|k(s)(t)$ and such that the genus $g(\widehat{E})$ equals $g(E)$. 
In these cases, \cite[Lemma 8.2.4]{Jarden} yields the following: There exists an isomorphism 
$$\varphi_{s_0}:{\textrm{Gal}}(E|k(s)(t)) \to {\textrm{Gal}}(E_{s_0}|k(t)),$$
such that for each branch point $t\mapsto b$ (with $b$ in the algebraic closure of $k(s)$), the inertia group $I_b$ is mapped under $\varphi_{s_0}$ into the inertia group $I_{\widehat{b}}$ at $\widehat{b}$ in $\widehat{E}$. 
Now choose $s_{0}$
such that the images of the branch points of $E|k(s)(t)$ under constant reduction are pairwise distinct. Obviously this excludes only finitely many more $s_0$. 
Then the inertia groups $I_{\widehat{b}}$ cannot be of strictly larger order than $I_b$, or otherwise $g(\widehat{E})>g(E)$. 
Therefore $b\mapsto \widehat{b}$ is a bijection from the set of branch points of $E|k(s)(t)$ to the set of branch points of 
$\widehat{E}|k(t)$ 
such that the respective inertia groups are isomorphic under $\varphi_{s_0}$.

Furthermore, ${\textrm{Gal}}(E_{s_0}|k(t))$ is canonically isomorphic (for all non-branch points $s_0\in k$) to the Galois group of the completion $E_{s_0} \cdot k(t)((s-s_0))|k(t)((s-s_0))$, 
where the compositum is uniquely determined up to embedding of $E_{s_0}|k(t)$ into an algebraic closure of $k(t)((s-s_0))$. Fix such an embedding for one given value $s_{0,0}$; 
then embeddings for all other values $s_0\in k$ are determined by continuous paths in $\mathbb{P}^1$. Since closed paths induce a simultaneous conjugation action of (the monodromy group) $G$
on the tuple of inertia group generators, 
the resulting identifications of ${\textrm{Gal}}(E_{s_0}|k(t))$ and ${\textrm{Gal}}(E|k(s)(t))$ are unique up to conjugation.
Thus, up to identifying ${\textrm{Gal}}(E|k(s)(t))$ and ${\textrm{Gal}}(E_{s_{0,0}}|k(t))$, all but finitely many values $s_0\in k$ in fact yield 
inertia group $I_i$ at the branch point $p_i(s_0)$ up to conjugation in $G$, for all $i=1,...,r$.

This shows the assertion.
\end{proof}
%


In order to avoid technicalities, we have not excluded cases which should not really be considered as ``families" in Definition \ref{def:family}; e.g.\ extensions given by polynomials constant in $s$ (in which case $E|k(s)(t)$ is isotrivial in the sense that it can be defined over $k(t)$); 
or extensions given by things like $f(X)-(t+s)$, in which case $E|k(s)(t)$ can be defined over $k(t+s)$. In this last case, the regular extensions arising from specializing $s\mapsto s_0$ are ``weakly equivalent":
They only differ by linear transformations in the variable $t$.\\
It should be noted that such ``trivial cases" automatically do not satisfy the assumptions of the theorems below, especially Theorem \ref{dec_gr_families}.

\subsection{Background on Hurwitz spaces and universal families}
Non-trivial one-parameter families arise naturally in inverse Galois theory, via rational curves on Hurwitz spaces. The following is a brief introduction into this subject; see e.g.\ \cite{FV} or \cite{RW} for more in-depth introductions.

Let $G$ be a finite group, $\mathcal{C}:=(C_1,...,C_r)$ be a $k$-rational $r$-tuple of conjugacy classes of $G$. Assume that there exists $(\sigma_1,...,\sigma_r)\in C_1\times...\times C_r$ with $\langle \sigma_1,...,\sigma_r\rangle=G$ and $\sigma_1\cdots\sigma_r=1$.

Riemann's existence theorem then asserts, for any $r$-set of branch points $p_1,...,p_r\in \mathbb{P}^1(\cc)$, the existence of a Galois covering with group $G$ and ramification type $((p_1,...,p_r),(C_1,...,C_r))$.
Via an equivalence relation, induced by isomorphism of the covering manifolds, the set of all such Galois coverings can be turned into a topological manifold, and moreover into an $r$-dimensional quasi-projective algebraic variety, commonly denoted by $\mathcal{H}^{in}(\mathcal{C})$, the (inner) Hurwitz space of $\mathcal{C}$. Cf.\ e.g.\ \cite[Section 1.2.]{FV}, \cite{RW}, or also \cite[Section 2.3]{thesis} for detailed introductions, as well as for several slight variants of the spaces $\mathcal{H}_r^{in}(G)$.

A famous result by Fried and V\"olklein (see \cite{FV}; and \cite[Thm.\ 1.7]{DF} for the version below) then links the regular inverse Galois problem over $k$ to existence of $k$-rational points on Hurwitz spaces:
\begin{theorem}
\label{hurwitz_points}
Let $G$ be a finite group with $Z(G)=1$, and let $\mathcal{C}$ be an $r$-tuple of conjugacy classes of $G$ such that $\mathcal{H}^{in}(\mathcal{C})$ is non-empty and connected.
\\
There is a universal family $\mathcal{F}:\mathcal{T}(\mathcal{C}) \to \mathcal{H}^{in}(\mathcal{C})\times \mathbb{P}^1(\cc)$ of ramified coverings, such that for each $h\in \mathcal{H}^{in}(\mathcal{C})$, the fiber cover $\mathcal{F}^{-1}(h) \to \mathbb{P}^1(\cc)$ is a ramified Galois cover with group $G$ and inertia canonical invariant $\mathcal{C}$.\\
This cover is defined regularly over a field $k \subseteq \cc$ if and only if $h$ is a $k$-rational point.
\end{theorem}
In particular, for a $k$-rational class tuple $\mathcal{C}$ as in Theorem \ref{hurwitz_points}, the family $\mathcal{F}$ gives rise to a Galois extension $F|k(\mathcal{H})(t)$ with group $G$, where $k(\mathcal{H})$ is the function field of the Hurwitz space $\mathcal{H}:=\mathcal{H}^{in}(\mathcal{C})$. Now if this space contains not only a $k$-rational point, but a rational curve defined over $k$, then restriction to this curve yields a Galois extension $E|k(s)(t)$ with group $G$.\\
There are however well-known theoretical criteria guaranteeing (in some cases) the existence of such rational curves, via the action of the Hurwitz braid group. One such criterion is contained in Thm. III.7.8 of \cite{MM}. These criteria have yielded existence of one-parameter families of regular $G$-extensions $E|\qq(t)(s)$ for many ``small" simple groups, and several papers have been dedicated to explicitly parameterizing such families (sample papers are \cite{Malle} and \cite{KoeHurwitz}).

\subsection{Decomposition groups in specializations of one-parameter families}
In what follows, let $E|k(s)(t)$ be a one-parameter family of regular $G$-Galois extensions of $k(t)$. We investigate the behaviour of local extensions at ramified primes in number fields arising from $E|k(s)(t)$ via specialization of {\textit {both}} variables. In particular, we investigate to what extent a prescribed pair $(I,D)$ can be obtained as inertia and decomposition group in specializations $E_{s_0,t_0}|k$ at a prescribed prime of $k$.
Our main goal in this section is Theorem \ref{dec_gr_families}. It shows that, under certain additional technical assumptions, the situation in specializations of one-parameter families becomes considerably richer than for specialization of single regular extensions. See Remark \ref{rem:compare} for a comparison with the situation of one single regular extension.

We begin with a lemma about behaviour of residue fields and decomposition groups under evaluation of the extra parameter $s$. Note that, since we defined inertia and decomposition groups in $E|k(s)(t)$ with respect to the parameter $t$ 
(over the constant field $k(s)$), we should view evaluation $s\mapsto s_0\in k$ as a constant reduction, and denote the reduced function field extension by $E_{s\mapsto s_0}|k(t)$.
\begin{lemma}
\label{decomp_field_reduction}

Let $E|k(s)(t)$ be a one-parameter family of regular Galois extensions with group $G$, with a $k(s)$-rational branch point $t\mapsto m(s)\in k(s)$. Let $F|k(s)$ be the residue field extension at $t\mapsto m(s)$ in $E$, and let $D$ (resp., $I$)
denote the decomposition group (resp., inertia group) at $t\mapsto m(s)$ in $E$.\\ 
Then for almost all $s\mapsto s_0\in k$, the following hold:
\begin{itemize}
\item[a)] The residue field $R$ of places extending the ``reduced" place $t\mapsto m(s_0)$ in $E_{s\mapsto s_0}|k(t)$ equals $F_{s_0}|k$.
\item[b)] The decomposition group $D_0$ at $t\mapsto m(s_0)$ in $E_{s\mapsto s_0}|k(t)$ is a subgroup of $D$ containing $I$ and such that $D_0/I$ equals $Gal(F_{s_0}|k)$ up to canonical isomorphism.
\end{itemize}
\end{lemma}
\begin{proof}
We first show:\\
{\textbf{Claim 1}}: $F_{s_0}\subseteq R$.

In the setting of Prop.\ \ref{mainlemma_kln}, take $k(s)$ as the base field, $t_0:=m(s_0)$, $t_i:=m(s)$ and $\mathfrak{p}$ the ideal of $k[s]$ generated by $s-s_0$.
Then $I_\mathfrak{p}(t_0,t_i)>0$, and so $(E_{t_0})_{s_0}$ (which is just the residue field of $E_{t_0}|k(s)$ at $\mathfrak{p}$) fulfills $(E_{t_0})_{s_0}\supseteq (E_{t_i})_{s_0} = F_{s_0}$ by Prop.\ \ref{mainlemma_kln}i). 
Now denote by $\nu$ a prolongation to $E$ of the valuation induced by $\mathfrak{p}$. The theory of constant reduction asserts that for almost all $s_0$, the field $E_{\nu}$ equals the residue field of $E_{s_0}$ of $E|k(t)(s)$ at 
(a place extending) $(s-s_0)\subset k(t)[s]$. Thus, for almost all $s_0\in k$, one has $(R=) (E_{\nu})_{t_0} = (E_{s_0})_{t_0}$. 
To show Claim 1, it therefore suffices to verify that the two specializations commute, which is of course always the case, as $(E_{t_0})_{s_0} = (E_{s_0})_{t_0}$ is the residue field of a point over $(s_0,t_0)$ of the algebraic variety 
corresponding to $E|k(s,t)$. Therefore $F_{s_0}\subseteq R$.

Now if $m(s)\in k$ is a constant, then $t_0=t_i$ and equality $F_{s_0}=R$ is thus obvious from the above. But otherwise, there are only finitely many $s_0\in k$ such that $I_\mathfrak{p}(t_0,t_i)\ne 1$ 
(Otherwise $(s-s_0)^2$ would have to divide $m(s)-m(s_0)$, and as $m(s)-m(Y)\in k[s,Y]$ is separable, this can only happen at finitely many $s_0$ (roots of the discriminant)). 
So by Beckmann's theorem, for all but finitely many $s_0$, the inertia group at $s\mapsto s_0$ in $E_{t_0}|k(s)$ has the maximal possible order (namely the ramification index of $t\mapsto m(s)$ in $E|k(s)(t)$), 
and now equality $F_{s_0} = (E_{t_i})_{s_0} = (E_{t_0})_{s_0} = R$ follows from Prop.\ \ref{mainlemma_kln}iii).

As for b), the containment of $I$ as the inertia group is clear from Lemma \ref{nondeg_spec}. Then, it is well known that $D_0/I$ is canonically isomorphic to the Galois group of the residue field extension $R|k$.
\end{proof}
We will now investigate local behaviour in specializations of one-parameter families, under some relatively mild assumptions. These include e.g.\ that the residue field extension $F|k(s)$ at some branch point is ``somewhat close" to being regular over $k$. 
Note that we cannot expect $F|k(s)$ itself to be regular in general, as $\zeta_n \in F$ where $n$ is the ramification index at the branch point in question (see e.g.\ Lemma 2.3 in \cite{KLN}).
\begin{theorem}
\label{dec_gr_families}
Let $E|k(s)(t)$ be a regular Galois extension with group $G$, and let $t\mapsto m(s)\in k(s)$ be a $k(s)$-rational branch point. 
Let $I\trianglelefteq D\le G$ denote an inertia and decomposition group at this branch point. 
Assume that furthermore the residue field extension $F|k(s)$ over $t\mapsto m(s)\in k(s)$ in $E$ is not a constant field extension of $k(s)$.
Then the following hold:
\begin{itemize}
\item[i)] 
For {\textit {all but finitely many}} primes $p$ and for all subgroups $U\le I$, there are infinitely many $s_0,t_0\in k$ such that the inertia group at a prime extending $p$ in $E_{s_0,t_0}|k$ equals $U$ (up to conjugation in $G$) 
and the decomposition group is not contained in $I$.
\item[ii)] For all but finitely many primes $p$ which split completely in $F_0:=F\cap \overline{\qq}$, for all $x\in Gal(F|F_0(s)) (\le D/I)$, and all subgroups $U\le I$, the following holds:\\
There are infinitely many $s_0,t_0\in k$ such that the inertia group at a prime extending $p$ in $E_{s_0,t_0}|k$ equals $U$ and the decomposition group 
is mapped onto $\langle x\rangle$ under the canonical epimorphism $D\to D/I$. \footnote{Note that, of course, in the important special case of ``full inertia", i.e.\ $U=I$, the subgroup of $D$ with these properties is unique.}
\end{itemize}
Finally, assume that $F|k(s)$ even contains a non-trivial $k$-regular Galois subextension $\tilde{F}|k(s)$. Then:
\begin{itemize}
\item[iii)] For all but finitely many primes $p$ of $k$, for all $x\in Gal(\tilde{F}|k(s))$ and all subgroups $U\le I$, the conclusion of ii) holds, when the epimorphism $D\to D/I$ is replaced by $D\to D/N$ 
with $N\triangleleft D$ the Galois group of $E\cdot k(s)((t-m(s)))/ \tilde{F}((t-m(s)))$.
\end{itemize}
\end{theorem}

\begin{remark}
\label{rem:compare}
We briefly compare the different statements of Theorem \ref{dec_gr_families} with the case of one single $k$-regular extension of $k(t)$.
\begin{itemize}
\item[a)] The assertion of i) can never be reached via specializing one extension $E|k(t)$. Indeed, 
as a consequence of Prop.\ \ref{mainlemma_kln}, there exist infinitely many primes $p$ of $k$ such that decomposition groups at {\textit{all}} specializations $E_{t_0}|k$ in which $p$ ramifies are contained inside the respective (geometric) inertia group of $E|k(t)$. See Proposition 6.3 in \cite{KLN}.
\item[b)] Weak analogs of the assertions of ii) and iii) hold in the case of one extension $E|k(t)$; however only for infinite sets of primes $p$. As in a), the same results for {\textit{all}} primes, or all primes which are totally split in some prescribed finite extension of $k$, are impossible.
\end{itemize}
\end{remark}

\begin{proof}[Proof of Theorem \ref{dec_gr_families}.]
Our general goal is to first specialize the parameter $s$ suitably in order to obtain a single $k$-regular extension with a ``good" decomposition and inertia group at some branch point, and then to apply Proposition \ref{mainlemma_kln}.

Firstly, by a linear transformation in the variable $t$, we can and will assume all branch points to be finite and integral over $O_k[s]$, so $m(s)\in O_k[s]$.

Let $p$ be a given prime of $k$, not in some explicit finite set of primes depending only on $E|k(s)(t)$ (to be specified by the following).
Denote the residue field of a place extending the ramified place $t\mapsto m(s)$ in $E|k(s)(t)$ by $F|k(s)$. 
By Lemma \ref{decomp_field_reduction}, we may assume, up to excluding finitely many $s_0\in k$, that $F_{s_0}|k$ is the residue field at $t\mapsto m(s_0)$ in $E_{s_0}|k(t)$ and that the decomposition group $D_0$ at $t\mapsto m(s_0)$ fulfills $D_0/I = {\textrm{Gal}}(F_{s_0}/k) \le D/I$.


Assume now (as in i)) that $F|k(s)$ is not a pure constant field extension, and let $F_0$ be the exact field of constants. Now there are only two possibilities: 

{\textit {Case 1}}: $p$ does not split completely in $F_0|k$. Since every specialization $F_{s_0}|k$ contains the constant field $F_0$, we know that $F_{s_0}|k$ has non-trivial Frobenius element at $p$ (for all $s_0\in k$). 
But then by Prop.\ \ref{mainlemma_kln}, for all but finitely many primes $p$ and for all $t_0\in k$ fulfilling $ord_p(t_0-m(s_0))>0$, the inertia group at a prime extending $p$ in $E_{t_0,s_0}|k$ is contained in $I$, whereas the image under canonical projection $D\to D/I$ of a decomposition group is nontrivial. Note that the finite set of primes which are excluded from this conclusion depends a priori on $s_0$! However, it will be shown below that this dependency can be removed.


{\textit {Case 2}}: $p$ splits completely in $F_0|k$. \footnote{Note that for those primes, assertion ii) implies i).} Let $p'$ be a prime extending $p$ in $F_0$. In particular, $O_{F_0}/p'$ is canonically isomorphic to $O_k/p$. 
Let $F|F_0(s)$ be the splitting field of the absolutely irreducible polynomial $\rho(s,X)$ (of degree $>1$ and without loss of generality in $O_k[s,X]$).\\
For all but finitely many of these $p$, $\rho(s,X)$ is defined over $O_k/p$, and after excluding finitely many more $p$, we can assume that $\rho$ even remains absolutely irreducible over $O_k/p$, with the same Galois group as $F|F_0(s)$. 

We are looking for specializations $s\mapsto s_0$ such that $Frob_{p'}(F_{s_0}|F_0)$ 
equals a prescribed element $x\in Gal(F|F_0(s))$ (up to conjugation). Note here that, since $p$ is completely split in $F_0|k$, the completions of $F_{s_0}|F_0$ at $p'$ and of $F_{s_0}|k$ at $p$ are canonically isomorphic.
Since $F|F_0(s)$ is $F_0$-regular, the existence of such specializations follows (after excluding finitely many $p$) from \cite[Theorem 1.2]{DG}; see also the proof of Prop.\ 5.1. in \cite{Debes}.
Here, we can even assume $s_0\in k$ (not just $s_0\in F_0$); this is obvious from the proof of \cite[Theorem 1.2]{DG}, but since it is not explicitly contained in the statement, we briefly recall the argument:\\
Take a Galois cover $f_p:X\to \mathbb{P}^1(k) \otimes_k k_p$ corresponding to the constant extension $F \cdot k_p|k_p(s)$ of $F|F_0(s)$ by the complete field $k_p$ (with Galois group canonically isomorphic to ${\textrm{Gal}}(F|F_0(s))$, since the latter extension is $F_0$-regular),
and consider the unique unramified epimorphism $\varphi_p:G_{k_p}\to \langle x\rangle$ sending the Frobenius of $k_p$ to $x$. 
Then the existence of $k_p$-points in the fiber over $s_0$ in the ``twisted cover" of $f_p$ by $\varphi_p$ is sufficient for $Frob_p(F_{s_0}|k)$ to equal $x$ (up to conjugacy). 
Now the mod-$p$ reduction of $f_p$ corresponds to the regular extension generated by the roots of $\rho$ over $O_k/p$. By Hensel's lemma, every unramified rational point of this mod-$p$ reduction lifts to a $k_p$-point of $f_p$. 
By the Lang-Weil theorem, the existence of such mod-$p$ points is guaranteed for all but finitely many $p$ (and in fact the number of such mod-$p$ points is $O(N(p))$). 
Hensel lifting then yields infinitely many integral specializations $s\mapsto s_0\in O_k$ (in fact, arithmetic progressions) with the prescribed Frobenius $Frob_p(F_{s_0}|k)$.
By Prop.\ \ref{mainlemma_kln}ii), this yields for all $t_0$ fulfilling $ord_p(t_0-m(s_0))>0$, that the decomposition group at $p$ in $E_{s_0,t_0}|k$ is of the form $\langle I_{p,s_0,t_0},\widehat{x}\rangle$, 
where $I_{p,s_0,t_0}$ is the inertia group and $\widehat{x}$ is a suitable preimage of $x$ under the canonical projection $D\to D/I$.

Note that we have to make sure that after specialization $s\mapsto s_0\in O_k$ as above, the prime $p$ has not become an exceptional prime (in the sense of Prop.\ \ref{mainlemma_kln}) in the extension $E_{s_0}|k(t)$. The following lemma asserts that this can be achieved.
\begin{lemma}
\label{lemma:no_bad_prime}
There exists an integer $m\in \nn$, only depending on $E|k(s)(t)$ (and not on the prime $p$!) such that the set of values $s_0\in O_k$ for which $p$ becomes an exceptional prime of $E_{s_0}|k(t)$ is contained in a union of at most $m$ residue classes mod $p$. 
\end{lemma} 
\begin{proof}
From the proof of Prop.\ \ref{mainlemma_kln} in \cite{KLN}, the exceptional primes are either bad primes in the sense of Beckmann's theorem, or contained in one of four further explicitly given finite sets $\mathcal{S}_i$, $i=1,\dots, 4$. First consider the bad primes of Beckmann's theorem.
Since $p$ can be assumed without loss of generality to not divide $|G|$, there are only three further types of bad primes (cf.\ Definition 2.6 and the subsequent ``Specialization Inertia Theorem" in \cite{Legrand}):
\begin{itemize}
\item[a)] Primes that have vertical ramification in $E_{s_0}|k(t)$.
\item[b)] Primes at which two branch points of $E_{s_0}|k(t)$ meet. 
\item[c)] Primes $p$ such that some branch point of $E_{s_0}|k(t)$ is not $p$-integral.
\end{itemize}
Case c) will not occur since we assumed all branch points to be integral over $O_k[s]$ and specialized $s\mapsto s_0\in O_k$.
Condition b) can be transformed into $s_0$ being a root modulo $p$ of at least one of finitely many polynomials of bounded degree (not depending on $s_0$ or $p$). Obviously, there is at most a bounded number of such roots modulo $p$.

Finally, to exclude vertical ramification, we may use the following criterion, mentioned e.g.\ in \cite[Addendum 1.4c)]{DG}:\\
Let $P(t,X)$ be the minimal polynomial of a primitive element $\xi$ of a regular Galois extension $E|k(t)$, and assume $P\in O_k[t,X]$ monic in $X$. Then $E|k(t)$ has no vertical ramification at $p$ if the discriminant $\Delta(P)\in k[t]$ (with regard to $X$) is $\ne 0$ mod $p$.
In our situation, take $P(s,t,X)$ to be the minimal polynomial of a primitive element $\xi$ of $E|k(s)(t)$, and assume without loss of generality (after linear transformation in $\xi$, if necessary) that $P\in O_k[s,t,X]$ is monic in $X$. Firstly we can assume that $p$ does not divide the discriminant $\Delta(P)\in O_k[s,t]$.
But then, vertical ramification at $p$ in $E_{s_0}|k(t)$ (for $s_0\in O_k$) can only happen if $(s_0$ mod $p) \in O_k/p$ is a root of $\Delta(P)$ (viewed as a polynomial in $s$ over $(O_k/p)[t]$). Since $\Delta(P)\in (O_k/p)[t][s]$ is not $0$ mod $p$ and has a bounded degree (independent of choice of $p$ or $s_0$), this can only happen for $s_0$ in a bounded number of mod-$p$ residue classes, as claimed.

Next, consider the additional exceptional sets $\mathcal{S}_i$ in the proof of Prop.\ \ref{mainlemma_kln}. Due to our assumption of a $k(s)$-rational branch point, the sets $\mathcal{S}_1$ and $\mathcal{S}_3$ are automatically empty by definition, whereas $\mathcal{S}_4$ is the set of primes modulo which the discriminant of the minimal polynomial of a primitive element of $E_{s_0}/k(t)$ becomes inseparable - this set can be dealt with just like case b) above. Finally, the set $\mathcal{S}_2$ contains the primes dividing denominators of coefficients of a fixed finite set of Puiseux expansions (namely, of a primitive element of $E_{s_0}/k(t)$ and of its intermediate fields). Again, these denominators are specializations of denominators $d(s)\in k[s]$ belonging to primitive elements of $E/k(s,t)$ and intermediate fields. Therefore once again, it suffices to choose $s_0\in O_k$ outside of the mod-$p$ roots of any of finitely many polynomials. This concludes the proof of the lemma.
\end{proof}

Due to Lemma \ref{lemma:no_bad_prime}, we can ensure, for all $p$ with sufficiently large residue field, that Lang-Weil above yields sufficiently many points mod $p$ in order to guarantee that at least one of them has $s$-value not in the set of exceptional residue classes of Lemma \ref{lemma:no_bad_prime}. We have therefore shown:\\
{\textbf {Claim 1:}} For all but finitely many primes $p$ of $k$ (with the exceptions only depending on $E|k(s)(t)$), there exist specializations $s\mapsto s_0\in O_k$ such that the prime $p$ is a good prime for the resulting regular Galois extension $E_{s_0}|k(t)$. Moreover, the set of such $s\mapsto s_0$ contains full cosets modulo $p$.

Note also again that by Lemma \ref{nondeg_spec}, we can assume that the place $t\mapsto m(s_0)$ of $E_{s_0}|k(t)$ has the same inertia group $I$ as the place $t\mapsto m(s)$ in $E|k(s)(t)$. Let $n:=|I|$. Beckmann's theorem then yields that for any divisor $\tilde{n}|n$ and for $t_0$ with $ord_p(t_0-m(s_0))=n/\tilde{n}$, the inertia subgroup of $p$ in $E_{s_0,t_0}|k$ equals the subgroup $U$ of $I$ of order $\tilde{n}$.\\
Together with the above results about the decomposition groups, the claims of i) and ii) follow.

Finally, the proof of iii) is completely analogous to ii), just with the $F_0$-regular extension $F|F_0(s)$ replaced by the $k$-regular extension $\tilde{F}|k(s)$.
\end{proof}

\section{Problems of Hilbert-Grunwald type}
\label{sec:hilbertgrunwald}
We will now derive results on prescribing local behaviour at finitely many (ramified and unramified!) primes at a time for specializations of a family $E|\qq(s)(t)$ of regular Galois extensions.
This may be seen as a generalization of the ``Hilbert-Grunwald" type theorems proved in \cite{DG} and \cite{Debes}.
%
\subsection{Main result}
Throughout this section, let $E|k(s)(t)$ be a one-parameter family of regular Galois extensions, and assume that all branch points are integral over $O_k[s]$; this assumption is without loss of generality via linear transformation in $t$.
\begin{theorem}
\label{simultaneous}
Assume additionally that there exist $k(s)$-rational branch points $t_1,...,t_r$ of $E|k(s)(t)$,
with inertia and decomposition groups $(I_1,D_1),...,(I_r,D_r)$, 
and such that the residue field extension over $t_i$ contains a regular Galois subextension $K_i|k(s)$, for $i=1,...,r$. Set $G_i:={\textrm{Gal}}(K_i|k(s))$. 
(One then has $G_i=D_i/N_i$ for some normal subgroup $I_i\triangleleft N_i\triangleleft D_i$.)\footnote{Note the important special case $N_i=I_i$ in case the whole residue field extension is regular. In this case, we retrieve the special
case of Theorem \ref{weak_version}.}

Let $S_1$ and $S_2$ be disjoint finite sets, both disjoint to some finite set of ``bad" primes only depending on $E|k(s)(t)$. 
\begin{itemize}
\item For $p\in S_1$ let $C(p)$ be a conjugacy class of $G$.
\item For $p\in S_2$ let $i:=i_p\in \{1,...,r\}$, and let $(A_i,x_i)$ be such that $1\ne A_i$ is a subgroup of $I_i$ and $x_i$ is some element of $G_i$.
\end{itemize}
 Then there exist infinitely many specializations $E_{s_0,t_0}|k$, with $s_0,t_0\in O_k$ such that
\begin{itemize}
\item[a)] For all $p\in S_1$, the extension $E_{s_0,t_0}|k$ is unramified at $p$ with Frobenius element in class $C(p)$.
\item[b)] For all $p\in S_2$, the extension $E_{s_0,t_0}|k$ is tamely ramified at $p$ with inertia group $A_i$ and decomposition group $B_i$ (up to conjugation in $G$) fulfilling $\varphi_i(B_i) = \langle x_i\rangle$ with the canonical projection $\varphi_i:D_i\to D_i/N_i$.
\end{itemize}
More precisely, the sets of such $s_0$ and $t_0$ contain full arithmetic progressions in $O_k$. 
\end{theorem}

The main point of Theorem \ref{simultaneous} is to quantify the intuition that specialization of one-parameter families should provide more towards solving the Grunwald problem for a given group $G$ than just single regular Galois extensions. Of course, it should not be expected to obtain a complete answer to the problem via this approach (simply due to the difficulty of the regular inverse Galois problem). 
However, given suitable families of regular $G$-extensions, Theorem \ref{simultaneous} yields that certain (often non-cyclic) subgroups can be realized as decomposition groups in $G$-extensions of $k$ at all but finitely many primes of $k$! Compare once again the results of \cite{KLN}, where it is shown that this is {\textit {never}} true for $G$-extensions arising as specializations of a single regular extension, as soon as the subgroup in question is non-cyclic.

\begin{proof}[Proof of Theorem \ref{simultaneous}]
First, consider the primes $p\in S_2$. The claim for each individual prime $p$ follows directly from Theorem \ref{dec_gr_families}iii). 
It is important to note that, by the proof of Theorem \ref{dec_gr_families}, the set of specializations $s\mapsto s_0\in O_k$ that allow the specialized extension $E_{s_0}|k(t)$ to possess specializations with the prescribed inertia 
and decomposition group at some $p\in S_2$ is a congruence set, i.e.\ contains full residue classes mod $p$.
Chinese remainder theorem then yields that for any finite set $S_2:=\{p_1,...,p_m\}$ of primes (disjoint from the set of exceptional primes for Theorem \ref{dec_gr_families}), the set of specializations of $s$ that are ``good" for all $p_i$ simultaneously (in the above sense) still contains full residue classes mod $\mathfrak{J}$ (with $\mathfrak{J}=\prod_{i=1}^m p_i$).

Note also once again that for the field $E_{s_0}|k(t)$ obtained in this way and for any $p\in S_2$, all specializations $t\mapsto t_0$ which fulfill $I_{p}(t_0,t_i)=d_i:=[I_i:A_i]$
yield the prescribed pair of inertia and decomposition group at $p$.
But again, the conditions $I_{p_i}(t_0,t_{i})=d_i$ (for finitely many primes $p_i$, $k$-rational branch points $t_{i}$ and integers $d_i$) can be fulfilled simultaneously for an arithmetic progression of values $t_0$. 
This shows part b) of the assertion.


Furthermore, for the regular $G$-extensions $E_{s_0}|k(t)$ obtained above, assume that $p\in S_1$ is {\textbf {not a bad prime}} in the sense of Beckmann's theorem, and let $C(p)$ a conjugacy class of $G$. Then by \cite{DG}, the set of $t_0\in O_k$ such that $E_{s_0,t_0}|k$ is unramified at $p$ with Frobenius in class $C(p)$ is a union of cosets mod $p$. Moreover, this union is non-empty if $p$ is not in some explicit finite set of exceptional primes - which depends only on $|G|$, the number of branch points of $E_{s_0}|k(t)$ and the genus of $E_{s_0}$. Via excluding finitely many $s_0\in k$, we may assume without loss of generality that the extension $E_{s_0}|k(t)$ is a non-degenerate member of the family $E|k(s)(t)$; see Lemma \ref{nondeg_spec}. 
Then the number of branch points of $E_{s_0}|k(t)$ equals the number of branch points of $E|k(s)(t)$ (i.e.\ is independent of the choice of $s_0$), 
and also the orders of inertia subgroups are the same as the ones at the respective branch points of $E|k(s)(t)$, whence the genus is also independent of $s_0$. Therefore, the set of exceptional primes can be chosen to depend only on $E|k(s)(t)$.
Since once again we have obtained mod-$p$ congruence conditions (for all $p\in S_1$), there exist arithmetic progressions of $t_0\in O_k$ fulfilling them simultaneously for all $p\in S_1$, and also simultaneously with the conditions for the $p\in S_2$, by the Chinese Remainder Theorem.

What is left to show is that we can choose $s_0\in O_k$ such that (the above congruence conditions for all $p\in S_2$ are fulfilled and) none of the primes $p\in S_1$ become bad primes for $E_{s_0}|k(t)$.
By Lemma \ref{lemma:no_bad_prime}, this is achieved by additionally requiring $s_0$ to fulfill suitable congruence conditions mod all $p\in S_1$, under the sole condition that $S_1$ is disjoint from some finite set of primes of $k$ (depending only on $E|k(s)(t)$). 
\end{proof}

\begin{remark}[Discriminant (and other) estimates]
The fact that the sets of values $t_0$ and $s_0$ fulfilling the assertions of the above theorem contain arithmetic progressions yields immediately that the specializations $E_{s_0,t_0}|k$ can always be chosen such that they have the full Galois group $G=Gal(E|k(s)(t))$, as a direct consequence of Hilbert's irreducibility theorem. Furthermore, our methods can be applied to yield lower bounds for the number of distinct $G$-extensions with the above properties and with discriminant of norm at most some prescribed integer $B$. Such bounds were obtained for specializations of a single regular extension, fulfilling only part a) of the previous theorem, in \cite[Theorem 1.1]{Debes}. We sketch briefly how to obtain such results in the context of Theorem \ref{simultaneous}. For simplicity, set $k=\qq$. Let $f(s,t,X)$ be the minimal polynomial of a primitive element of of $E|\qq(s)(t)$. Without loss of generality, we may assume $f(s,t,X)\in\zz[s,t,X]$ to be monic in $X$. Up to replacing $s$ and $t$ by suitable values $m_1 s + m_2$ and $n_1 t + n_2$, with $m_i, n_i\in \zz$ (depending of course on the sets $S_1$ and $S_2$), we can assume that all integer specializations of $s$ and $t$ which preserve the Galois group fulfill the assertions of Theorem \ref{simultaneous}. By Hilbert's irreducibility theorem, ``most values $(s_0,t_0)\in \zz^2$ satisfy this, e.g.\ the number of $(s_0,t_0)\in \zz^2$ with $|s_0|, |t_0|\le B$ is at most $B^{3/2+\epsilon}$ for sufficiently large $B\in \nn$ (see \cite{Cohen81}). Fixing one specialization value for $s$, say $s_0=m_2$ as above, Theorem 1.3 in \cite{Debes} yields that, for sufficiently large $B$, these values $(s_0,t_0)$ lead to at least $B^{1-1/|G|-\epsilon}$ different number fields with Galois group $G$ and fulfilling the assertions of Theorem \ref{simultaneous}. Furthermore, all these number fields have discriminant of absolute value at most $|\Delta(f(m_2,t_0,X))| \ll C_1\cdot B^{C_2}$. Here the constant $C_2>0$ depends only on $f$, i.e.\ on $E|\qq(s)(t)$, whereas $C_1>0$ depends also on $m_2$, and therefore on the sets of primes $S_1$ and $S_2$. It would be interesting to obtain stronger estimates by allowing $s_0$ to vary.
\end{remark}

\begin{remark}
The assumption of $k(s)$-rational branch points in Theorem \ref{simultaneous} of course cannot be fulfilled for all groups. E.g., if $G$ is an abelian group and $C$ a conjugacy class of elements of order $n$ in $G$, then the residue field $k(t_i)$ of a branch point $t_i$ with inertia group generator in class $C$ always has to contain $\zeta_n$, as a special case of Fried's branch cycle lemma.\\
%
Variants of Theorem \ref{simultaneous} taking such situations into account can easily be derived.
Just as an example, for a branch point $t\mapsto t_i\in \overline{\qq}\cup\{\infty\}$ of $E|k(s)(t)$, which is non-rational, but constant (in $s$), the above proofs show that conclusion b) of the above Theorem \ref{simultaneous} at least holds for all but finitely many primes $p$ which are completely split in $k(t_i)|k$ - while conclusion a) does not require the branch point condition and therefore still holds in full generality.
Note that this additional restriction on the primes $p$ is not necessarily an obstruction to solving Grunwald problems for $G$, since certain primes may not even ramify in any $G$-extension of $k$.
See Theorem \ref{thm:elem_ab} and its proof for an example.
\end{remark}

\subsection{Example: The group $PSL_3(2)$ and application to parametric sets}
\label{ex:psl32}
We demonstrate application of the above criteria with a sample one-parameter family with Galois group $PSL_3(2)$.
Let \\
$f(s,t,X):=X^7 - 2sX^6 + (s^3 + s^2 + 3s - 2)X^4 + (-2s^3 - 4s^2 + 5s - 8)X^3 +
    (s^3 + 4s^2 - 10s + 16)X^2 + (-s^2 + 5s - 12)X - s + 4 +tX^2(X - 1)(X^2 - sX + s)$.

Let $k$ be a number field, $K|k(s)(t)$ be a root field of $f$ and $E|k(s)(t)$ be the Galois closure. This extension has regular Galois group $PSL_3(2)$ 
(the polynomial $f$ is a specialization of a multi-parameter family given by Malle in \cite{Malle}). We are only interested in the inertia group and residue field extension at $t\mapsto \infty$. Note that $f$ is of $t$-degree $1$,
and therefore $K$ is a rational function field, 
say $K=k(s)(x)$. The splitting behaviour of $f$ shows that there among the places extending $t\mapsto \infty$ in $K$, there are two of ramification index $2$ and residue degree $1$ (namely, $x\mapsto 0$ and $x\mapsto \infty$, 
as well as two unramified places, of degree $2$ and $1$ respectively. In particular, the inertia group is of order $2$ (with cycle type $(2^2.1^3)$ in the degree-$7$ action of $PSL_3(2)$ on the roots of $f$), and the decomposition group 
has orbit lengths $2,2,2,1$ in the degree-$7$ action. From the subgroup structure of $PSL_3(2)$, one verifies quickly that this group then has to be a subgroup of $C_2\times C_2$, and in fact equality holds, since the
residue extension over $t\mapsto \infty$ contains a quadratic regular subextension $k(s,\sqrt{s^2-4s})|k(s)$. \\Application of Theorem \ref{simultaneous} then shows that, for any finite set $S$ of primes of $k$ (disjoint from a fixed finite set), 
there are $k$-specializations $s\mapsto s_0, t\mapsto t_0$ such that $E_{s_0,t_0}|k$ is a $G$-extension with inertia group $C_2$ and decomposition group $C_2\times C_2$ at all $p\in S$ (and of course, one can also replace this by 
``inertia group $C_2$ and decomposition group $C_2$" for any subset $S_0\subset S$, etc.).\\
By Remark \ref{rem:compare}a), such a phenomenon would be impossible when specializing only a single regular Galois extension $L|k(t)$.

In fact, the above observations suffice to yield a non-existence result about finite parametric sets for the group $PSL_3(2)$. Recall that a set $\mathcal{S}$ of 
$k$-regular $G$-extensions of $k(t)$ is called {\textit{parametric}} if every $G$-extension of $k$ occurs as a specialization of some element of $\mathcal{S}$, cf.\ \cite[Def.\ 7.1]{KLN}.
Non-existence of finite parametric sets over number fields was first proved, for many finite groups $G$, in \cite{KL}, and for the first family of simple groups (namely the alternating groups) in \cite{KLN}, Section 7.

Using the above results, we now obtain:
\begin{kor}
\label{psl32_nonparam}
Let $k$ be a number field and $G:=PSL_3(2)$. Then there is no finite parametric set for $G$ over $k$. 
\end{kor}
\begin{proof}
As shown above, for almost all primes $p$ of $k$, there exists a $G$-extension of $k$ with the abelian, but non-cyclic decomposition group $C_2\times C_2$ at $p$. The assertion now follows immediately from \cite[Theorem 7.2]{KLN}.
\end{proof}

It should be understood that the same argument is immediately applicable to further simple groups $G$, using families of regular extensions which exist in the literature, such as in \cite{Malle} or \cite{KoeHurwitz}.
The main point of our approach (which is new compared to the treatment in \cite{KLN}) is that the existence of one-parameter families with group $G$ (and with certain technical assumptions) 
contradicts the existence of finite parametric sets for $G$.

%
%

\subsection{Example: Elementary-abelian groups}
\label{ex:elab}
As a further example, we present a case where one-parameter families of regular extensions suffice to give, via specialization, a complete answer to Grunwald problems away from a finite set of primes.
\begin{theorem}
\label{thm:elem_ab}
Let $G=(C_p)^n$ be an elementary-abelian $p$-group and $k$ be a number field. Then there exists a one-parameter family $E|k(s)(t)$ of regular Galois extensions with group $G$ and a finite set $S_0$ of primes of $k$ such that the following holds:

Let $S$ be any set of primes of $k$, disjoint from $S_0$, and for each $\nu\in S$ let $(I_\nu,D_\nu)$ be a pair of subgroups of $G$ and let $L_\nu|k_\nu$ be a local Galois extension with Galois group embedding into $G$, such that there exists an isomorphism $D_\nu \to {\textrm{Gal}}(L_\nu|k_\nu)$ which maps $I_\nu$ to the inertia subgroup of $L_\nu|k_\nu$.\\
Then there exist infinitely many specializations $E_{s_0,t_0}|k$ of $E|k(s)(t)$, still with Galois group $G$, whose completion at $\nu$ is $L_\nu|k_\nu$ and whose inertia and decomposition group at $\nu$ equal $(I_\nu,D_\nu)$ (for all $\nu\in S$). \footnote{Recall once again that this last equality is understood up to canonical identifications (and up to conjugacy in $G$, which however is trivial in this case), not just up to abstract isomorphism.}
\end{theorem}
\begin{proof}
As in the proof of Theorem \ref{simultaneous}, it suffices to consider one prime at a time, and it is not necessary to verify the property ${\textrm{Gal}}(E_{s_0,t_0}|k)=G$.

We first assume that we can construct the $k(s)$-regular Galois extension $E|k(s)(t)$ such that the following holds:
\begin{itemize}
\item[(*)] For each cyclic subgroup $I$ of $G$ there exists a $k(s)(\zeta_p)$-rational branch point of $E|k(s)(t)$ with inertia group $I$, decomposition group $G$, and $k(\zeta_p)$-regular residue field extension. 
\end{itemize}

We show the assertion under the above assumption. So let $\nu\in S$, and $(I_\nu,D_\nu)$ a pair of subgroups of $G$ and $L_\nu|k_\nu$ be a local extension whose inertia and decomposition group at $\nu$ equal 
$(I_\nu,D_\nu)$, such that $D_\nu$ is isomorphic to ${\textrm{Gal}}(L_\nu|k_\nu)$ and $I_\nu$ is isomorphic to the inertia subgroup of $L_\nu|k_\nu$. Up to excluding finitely many $\nu$, we may assume $L_\nu|k_\nu$ to be tame.
Of course it then holds that $|D_\nu|\le p^2$, since $I_\nu$ and $D_\nu/I_\nu$ must be cyclic. Since the unramified case is dealt with in \cite{DG}, we may and will assume $I_\nu$ to be non-trivial.
Thus $D_\nu = \langle I_\nu, x\rangle$ with $|I_\nu| = p$ and $x\in G$.
It suffices to show that, for suitable $s_0,t_0\in k$, the specialization $E_{s_0,t_0}|k$ has inertia and decomposition group $(I_\nu,D_\nu)$ at $\nu$. Due to the above assumptions on $E|k(s)(t)$, 
Theorem \ref{simultaneous} immediately yields this claim if one replaces the constant field $k$ by $k(\zeta_p)$ (and the prime $\nu$ by a prime extending it in $k(\zeta_p)$). 
However, note that $\nu$ must be completely split in $k(\zeta_p)|k$ (or else there would be no totally ramified $C_p$-extension of $k_\nu$). 
In particular, the values $s_0,t_0\in k(\zeta_p)$ obtained from application of Theorem \ref{simultaneous} can be assumed to lie in $k_\nu$. 
Of course they can then be approximated by values in $k$, and via Krasner's lemma we can assume the assertion for $E(\zeta_p)_{s_0,t_0}|k(\zeta_p)$, with $k$-rational values $s_0,t_0$. 
Furthermore, combining the previous argument with Hilbert's irreducibility theorem, we can additionally assume that $E_{t_0,s_0}|k$ has group $G$ and is linearly disjoint from $k(\zeta_p)|k$. 
Since $\nu$ is completely split in $k(\zeta_p)|k$, the completion $k_\nu$ equals the completion at any prime extending $\nu$ in $k(\zeta_p)$, up to canonical isomorphism. 
The local extension at $\nu$ in $E_{s_0,t_0}|k$ is then the same as the corresponding local extension in $E(\zeta_p)_{s_0,t_0}|k(\zeta_p)$. This shows the assertion.

It remains to show that a $k(s)$-regular extension $E|k(s)(t)$ with the above properties actually exists.
It is well-known that the elementary-abelian group $G$ possesses a generic extension (over $\qq$ and then a fortiori over $k(s)(t)$). In fact this follows directly from the analogous result about cyclic groups of prime order (see e.g.\ 
\cite[Corollary 5.3.3]{Ledet}). 
A result of Saltman (\cite[Thm.\ 5.9]{S82}) then implies the following: If $F_1|k(s)(t)_{\mathfrak{P}_1}$,..., $F_r|k(s)(t)_{\mathfrak{P}_r}$ are any given $r$ Galois extensions of completions of $k(s)(t)$ at distinct primes 
$\mathfrak{P}_1$,..., $\mathfrak{P}_r$, and $\phi_i: {\textrm{Gal}}(F_i|k(s)(t)_{\mathfrak{P}_i})\to D_i \le G$ are isomorphisms, then there exists a Galois extension $E|k(s)(t)$ with these prescribed local extensions and 
prescribed decomposition groups. Clearly, as soon as $\mathfrak{P}_i$ has residue field $k(s)(\zeta_p)$, there is a totally ramified $C_p$-extension of the completion at $\mathfrak{P}_i$ of $k(s)(t)$. 
Compositum of this extension with a $k$-regular $(C_p)^{n-1}$-extension of $k(s)$ 
yields the desired local extension.\\
Finally, it is easy to ensure that $E|k(s)(t)$ is $k(s)$-regular; in fact it suffices to have linearly disjoint (over $k(s)$) residue field extensions at two primes of $k(s)(t)$, which can be achieved via Saltman's result, just as above, adding more local extensions if necessary.
\end{proof}
\begin{remark}
Recall that of course Grunwald problems for $G=(C_p)^n$ have long been known to always have a solution, due to the Grunwald-Wang theorem and also due to the existence of generic extensions. It is, however, interesting to note that by the above, 
the necessary dimension (in the sense of transcendence degree over $k$) to solve Grunwald problems (away from some finite set of primes) via specialization is only $2$, whereas the generic dimension 
(that is, the minimal transcendence degree of a generic extension) is $\ge n$ for $(C_p)^n$ (see Corollary 8.2.14 and Proposition 8.5.2 in \cite{Ledet}). It should be an interesting object for further study to investigate the above 
{\textit{``Hilbert-Grunwald dimension"}} for more general groups $G$. 
At present, we can only add that this dimension is $\ge 2$ for many finite groups, but we do not know of any example where it is provably larger than $2$.
\end{remark}

\subsection{Conjectural implications for the Grunwald problem}
\label{sec:conj}
To conclude, we present a strong (of course, hypothetical) version of the regular inverse Galois problem.
We show that Theorem \ref{simultaneous} implies that this hypothesis would yield positive answers over many number fields $k$ to all Grunwald problems (Problems \ref{prob:grunwald} and \ref{prob:grunwald_groups}) away from a finite set of primes (depending on $k$ and the group $G$).

So let $G$ be a finite group, and assume the following hypothesis:\\
(H) There exists a one-parameter family $E|\overline{\qq}(s)(t)$ of $G$-extensions with only $\overline{\qq}(s)$-rational branch points, such that 
each conjugacy class of cyclic subgroups of $G$ occurs at least once as an inertia subgroup of $E|\overline{\qq}(s)(t)$, and such that the decomposition group at each branch point equals the full centralizer of the respective inertia group in $G$.

\begin{theorem}
\label{hypothetical}
Given a finite group $G$ fulfilling Hypothesis (H) above, there exists a number field $k_0$ such that for all number fields $k\supseteq k_0$, Problems \ref{prob:grunwald} and \ref{prob:grunwald_groups} have a positive answer as long as the set $S$ is disjoint from some fixed finite set of primes of $k$ (depending on $G$).
\end{theorem}
\begin{proof}
Of course $E|\overline{\qq}(s)(t)$ in Hypothesis (H) is defined over $k_0(s)(t)$ for some number field $k_0$, say as $\tilde{E}|k_0(s)(t)$. Up to increasing $k_0$, we may assume the following:
\begin{itemize}
\item[a)] All branch points are $k_0(s)$-rational,
\item[b)] the $|G|$-th roots of unity are contained in $k_0$.
\end{itemize} 
Let $k\supseteq k_0$ be a number field. 
Due to condition b), all inertia groups in $\tilde{E}k|k(s)(t)$ are central in the respective decomposition groups. Since decomposition groups at a given branch point in $\tilde{E}k|k(s)(t)$ cannot be smaller than in $E|\overline{\qq}(s)(t)$,
they must still equal the full inertia group centralizer, and the residue extensions must be $k$-regular over $k(s)$. 
Now let $S$ be any finite set of primes of $k$ (away from some finite exceptional set depending on $E|\overline{\qq}(s)(t)$), and for each $p\in S$ let $(A_p,B_p)$ be a pair of subgroups of $G$ such that $A_p$ is central in $B_p$ 
and $B_p/A_p$ is cyclic. Via choosing branch points of $\tilde{E}k|k(s)(t)$ with inertia group $A_p$, Hypothesis (H) above together with Theorem \ref{simultaneous} then imply that the induced Grunwald problem 
(Problem \ref{prob:grunwald_groups}) is solvable via specialization from $\tilde{E}k|k(s)(t)$. But again due to condition b), the decomposition group at {\textit{any}} tame Galois extension of the complete field $k_p$, 
with ramification index dividing $|G|$, centralizes the inertia group (and has cyclic quotient group).
So Problem \ref{prob:grunwald_groups} for the group $G$ always has a positive answer over $k$, with the exception of some finite set of primes of $k$. Finally, the analog for Problem \ref{prob:grunwald} follows via 
Prop.\ \ref{mainlemma_kln}iv).
\end{proof}

Note that the condition of existence of a one-parameter family as above is equivalent to existence of a rational genus-zero curve on a certain Hurwitz space (of a class tuple containing each conjugacy class at least once). As for the extra condition on ``large" decomposition groups, it is a ``generic" case that can in some sense be expected to be fulfilled most of the time (even though it is of course not easy to turn this intuition into a strict proof).

It should be noted that a simpler ``strong version" of the regular inverse Galois problem, also implying positive answer to all Grunwald problems, is the existence of a generic polynomial for the group $G$ (this implication is due to Saltman \cite{S82}). However, Saltman himself gave examples of groups which do not possess generic polynomials over any number field. Therefore this approach is too restrictive to solve the Grunwald problem in full. On the other hand, I am not aware of any group for which the above Hypothesis (H) is known to fail.

{\textbf{Acknowledgement}}:\\
I thank 
Pierre D\`ebes, 
Francois Legrand and Danny Neftin for helpful discussions.\\
This work was supported by the Israel Science Foundation (grant no. 577/15).


\begin{thebibliography}{9}
\bibitem{Beckmann} S.\ Beckmann, \textit{On extensions of number fields obtained by specializing branched coverings}. J.\ reine angew.\ Math.\ 419 (1991), 27--53.
\bibitem{Cohen81} S.\ D.\ Cohen, \textit{The distribution of Galois groups and Hilbert's irreducibility theorem}.
Proc.\ London Math.\ Soc. 43 (1981), 227--250.
\bibitem{Conrad} B.\ Conrad, \textit{Inertia groups and fibers}.
 J.\ Reine Angew.\ Math. 522 (2000), 1--26.
\bibitem{Debes} P.\ D\`ebes, \textit{On the Malle conjecture and the self-twisted cover}.
Isr.\ J.\ Math.\ 218 (2017), 101--131.
\bibitem{DF} P.\ D\`ebes, M.D.\ Fried, \textit{Arithmetic variation of fibers in families of curves. Part I: Hurwitz monodromy criteria for rational points on all members of the family}. J.\ Reine Angew.\ Math.\ 409 (1990), 106--137.
\bibitem{DG} P.\ D\`ebes, N.\ Ghazi, \textit{Galois covers and the Hilbert-Grunwald property}.
Ann.\ Inst.\ Fourier, 62 (3) (2012), 989--1013.
\bibitem{Deuring} M.\ Deuring, \textit{Reduktion algebraischer Funktionenk\"orper nach Primdivisoren des Konstantenk\"orpers}. Mathematische Zeitschrift 47 (1942), 643--654.
\bibitem{FV} M.\ Fried, H.\ V\"olklein, \textit{The inverse Galois problem and rational points on moduli spaces}. Math.\ Ann.\ 290 no.\ 4 (1991), 771--800.
\bibitem{Har07} D.~Harari, \textit{Quelques propri\'et\'es d'approximation reli\'ees \`a la cohomologie galoisienne d'un groupe alg\'ebrique fini}. Bull.\ Soc.\ Math.\ France, \textbf{135(4)} (2007), 549--564.
\bibitem{HW18} Y.~Harpaz, O.~Wittenberg, \textit{Z\'ero-cycles sur les espaces homog\`enes et probl\`eme de Galois inverse}. Preprint (2018). arxiv.org/abs/1802.09605.
\bibitem{Jarden} M.\ Jarden, \textit{Algebraic Patching}. Springer Monographs in Mathematics, Berlin-Heidelberg (2011).
\bibitem{Ledet} C.U.\ Jensen, A.\ Ledet, N.\ Yui,
\textit{Generic polynomials: Constructive Aspects of the Inverse Galois Problem}. Cambridge Univ.\ Press (MSRI Publications, Vol.\ 45) (2002).
\bibitem{thesis} J.\ K\"onig, \textit{The inverse Galois problem and explicit computation of families of covers of $\mathbb{P}^1\cc$ with prescribed ramification}. PhD thesis (2014).
\bibitem{KoeHurwitz} J.\ K\"onig, \textit{Computation of Hurwitz spaces and new explicit polynomials for almost-simple Galois groups}. Math.\ Comp.\ 86 (2017), 1473--1498.
\bibitem{KL} J.\ K\"onig, F.\ Legrand, \textit{Non-parametric sets of regular realizations over number fields}. J.\ Algebra 497 (2018), 302--336.
\bibitem{KLN} J.\ K\"onig, F.\ Legrand, D.\ Neftin, \textit{On the local behaviour of specializations of function field extensions}. Int.\ Math.\ Res.\ Not.\ IMRN Volume 2019, Issue 9 (2019), 2951--2980.
\bibitem{Legrand} F.\ Legrand, \textit{Parametric Galois extensions}. J.\ Algebra 422 (2015), 187--222.
\bibitem{Malle}
G.\ Malle,
\textit{Multi-parameter polynomials with given Galois group}.
J.\ Symb.\ Comput.\ 21 (2000), 1--15.
\bibitem{MM} G.\ Malle, B.H.\ Matzat, \textit{Inverse Galois theory}. Springer Monographs in Mathematics, Berlin-Heidelberg (1999).
\bibitem{RW} M.\ Romagny, S.\ Wewers, \textit{Hurwitz Spaces. Groupes de Galois arithm\'etiques et diff\'erentiels}.
S\'emin.\ Congr., vol.\ 13, Soc.\ Math.\ France, Paris (2006), 313--341.
\bibitem{S82} D.J.\ Saltman,
\textit{Generic Galois extensions and problems in field theory}.
Advances in Mathematics 43 (1982), 250--283.
\bibitem{Serre} J.-P.\ Serre, \textit{Topics in Galois Theory}. Jones and Bartlett, Boston, 1992.
\end{thebibliography}
\end{document}